\documentclass[12pt]{article}
\usepackage{enumerate}
\usepackage{graphicx}
\usepackage{amsmath}
\usepackage{amsthm}
\usepackage{amsfonts}
\usepackage{amssymb}
\usepackage{xcolor}
\usepackage[numbers]{natbib}
\theoremstyle{theorem}
\newtheorem{question}{Question}
\newtheorem{problem}[question]{Problem}
\newtheorem{conjecture}[question]{Conjecture}
\newtheorem{theorem}[question]{Theorem}
\newtheorem{proposition}[question]{Proposition}
\newtheorem{corollary}[question]{Corollary}

\newtheorem{lemma}[question]{Lemma}
\newtheorem{claim}[question]{Claim}
\newtheorem{remark}[question]{Remark}
\newtheorem{definition}[question]{Definition}
\newtheorem{construction}[question]{Construction}
\numberwithin{question}{section}
\numberwithin{equation}{section}

\def\eps{\varepsilon}

\title{Triangle-degrees in graphs and tetrahedron coverings in 3-graphs}
\author{Victor Falgas--Ravry\thanks{Ume{\aa} Universitet, Ume{\aa}, Sweden. Email: victor.falgas-ravry@umu.se} \and Klas Markstr\"om\thanks{Ume{\aa} Universitet, Ume{\aa}, Sweden. Email: klas.markstrom@umu.se}  \and Yi Zhao\thanks{Georgia State University, Atlanta GA, USA. Email: yzhao6@gsu.edu}}
\begin{document}
\maketitle
\begin{abstract}

We investigate a covering problem in $3$-uniform hypergraphs ($3$-graphs): given a $3$-graph $F$, what is $c_1(n,F)$, the least  integer $d$ such that if $G$ is an $n$-vertex $3$-graph with minimum vertex degree $\delta_1(G)>d$ then every vertex of $G$ is contained in a copy of $F$ in $G$ ?

We asymptotically determine $c_1(n,F)$ when $F$ is the generalised triangle $K_4^{(3)-}$, and we give close to optimal bounds in the case where $F$ is the tetrahedron $K_4^{(3)}$ (the complete $3$-graph on $4$ vertices).

This latter problem turns out to be a special instance of the following problem for graphs: given an $n$-vertex graph $G$ with $m> n^2/4$ edges, what is the largest $t$ such that some vertex in $G$ must be contained in $t$ triangles? We give upper bound constructions for this problem that we conjecture are asymptotically tight. We prove our conjecture for tripartite graphs, and use flag algebra computations to give some evidence of its truth in the general case.

\end{abstract}

\section{Introduction}

Let $F$ be a graph with at least one edge. What is the maximum number of edges $\mathrm{ex}(n,F)$ an $n$-vertex graph can have if it does not contain a copy of $F$ as a subgraph? This is a classical question in extremal graph theory. If $F$ is a complete graph, then the exact answer is given by Tur\'an's theorem~\cite{Turan41}, one of the cornerstones of extremal graph theory. For other graphs $F$, the value of $\mathrm{ex}(n,F)$ is determined up to a $o(n^2)$ error term by the celebrated Erd{\H o}s--Stone theorem~\cite{ErdosStone46}.

Ever since Tur\'an's foundational result, there has been significant interest in obtaining similar ``Tur\'an--type'' results for $r$-uniform hypergraphs ($r$-graphs), with $r\geq 3$. The extremal theory of hypergraphs has however turned out to be much harder, and even the fundamental question of determining the maximum number of edges in a $3$-graph with no copy of the tetrahedron $K_4^{(3)}$ (the complete $3$-graph on $4$ vertices) remains open --- it is the subject of a 70-years old conjecture of Tur\'an, and of an Erd{\H o}s \$ 1000 prize\footnote{In fact, to earn this particular Erd{\H o}s  prize, it is sufficient to determine the limit $\lim_{n\rightarrow \infty}\mathrm{ex}(n, K_t^{(r)}) /\binom{n}{r}$ for \emph{any} integers $t>r\geq 3$.}. Most of the research efforts have focussed on the case of $3$-graphs, where a small number of exact and asymptotic results are now known --- see \cite{BaberTalbot12, Bollobas74, deCaenFuredi00, FalgasRavryVaughan13, FurediPikhurkoSimonovits05}, as well as the surveys by F\"uredi ~\cite{Furedi91}, Sidorenko~\cite{Sidorenko95}, and Keevash~\cite{Keevash11}.

It is well-known that the Tur\'an problem for an $r$-graph $F$ is essentially equivalent to identifying the minimum vertex-degree required to guarantee the existence of a copy of $F$. More recently  \cite{CzNa, LoMarkstrom14, MuZh07},  there has been interest in variants where one considers what minimum \emph{$i$-degree} condition is required to guarantee the existence of a copy of $F$. 
Given an $i$-set $S\subseteq V(G)$ with $i\leq r$, its \emph{neighbourhood} in $G$ is the collection
\[\Gamma(S)=\Gamma_G(S) :=\{T \subseteq V(G)\setminus S: \ S \cup T\in E(G) \}\]
of $(r-i)$-sets $T$ whose union with $S$ makes an edge of $G$.  
The neighbourhood of $S$ defines an $(r-i)$-graph
\[G_S:= \left(V(G)\setminus S,  \Gamma_G(S)\right), \]
which is called the \emph{link graph} of $S$. 
The \emph{degree} of $S$ in $G$ is  the size $\deg_G(S)=\deg(S):=\vert \Gamma(S)\vert$ of its neighbourhood.  The \emph{minimum $i$-degree} $\delta_i(G)$ of $G$ is the minimum of $\deg(S)$ over all $i$-subsets $S\subseteq V(G)$. In particular,  the case $i=r-1$ has received particular attention; $\delta_{r-1}(G)$ is known as the minimum \emph{codegree} of $G$, and a minimum codegree condition is the strongest single degree condition one can impose on an $r$-graph. Determining what minimum codegree forces the existence of a copy of a fixed $r$-graph $F$ is known as the codegree density problem~\cite{MuZh07}. A few results on the codegree density for various small $3$-graphs are known, see~\cite{ FalgasRavryMarchantPikhurkoVaughan15, FalgasRavryPikhurkoVaughanVolec17+,KeZh,Mubayi05}.

In a different direction, there has been significant recent research activity devoted to generalising another foundational result in extremal graph theory. Let $F$ be a graph whose order divides $n$. What minimum degree condition is required to guarantee that a graph on $n$ vertices contains an \emph{$F$-tiling} --- a collection of $n/v(F)$ vertex-disjoint copies of $F$? In the case of complete graphs, this was answered by the celebrated Hajnal--Szemer\'edi theorem~\cite{HajnalSzemeredi70}, which (under the guise of equitable colourings) has applications to scheduling problems. For a general graph $F$, the K\"uhn--Osthus theorem~\cite{KuhnOsthus09}  determines the minimum degree-threshold for $F$-tilings up to a constant additive error.

There  has been a growing interest in determining analogous tiling thresholds in $r$-graphs for $r\geq 3$, see the surveys by R\"odl and Ruci\'nski \cite{RodlRucinski10}, and Zhao \cite{Zhao16} devoted to the subject.
In an effort to generalise Dirac's theorem on Hamilton cycles to hypergraphs, R\"odl, Ruci\'nski and Szemer\'edi~\cite{RodlRucinskiSzemeredi09} determined the minimum codegree threshold for the existence of a perfect matching in $r$-graphs for $r\ge 3$. The paper also introduced the hugely influential \emph{absorption method}, which has been used as a key ingredient in many of the results in the area obtained since.
Beyond perfect matchings, codegree tiling thresholds have by now been determined for a number of small $3$-graphs, including $K_4^{(3)}$~\cite{KeevashMycroft14, LoMarkstrom15},  $K_4^{(3)-}$~\cite{HanLoTreglownZhao15, LoMarkstrom13}, and $K_4^{(3)--}$ ($K_4^{(3)}$ with two edges removed)~\cite{CzygrinowDeBiasioNagle2014, KuhnOsthus06b}. In addition, the codegree tiling thresholds for $r$-partite $r$-graphs have been studied recently \cite{MR3567016, MR3656341, 2016arXiv161207247G, 2017arXiv170108115H,Mycroft16}

Turning to minimum vertex-degree tiling thresholds, fewer results are known. The vertex-degree thresholds for perfect matchings  were determined for $3$-graphs by Han, Person, and Schacht~\cite{HPS} (asymptotically) and by K\"uhn, Osthus and Treglown~\cite{KuhnOsthusTreglown13} and Khan~\cite{Khan13} (exactly). 
Han and Zhao~\cite{HanZhao15} determined the vertex-degree tiling threshold for $K_4^{(3)--}$, while Han, Zang, andZhao~\cite{HanZangZhao17} asymptotically determined the vertex-degree tiling threshold for all complete $3$-partite $3$-graphs.

As a key part of their argument, Han, Zang, and Zhao considered a certain $3$-graph covering problem and showed it was distinct from the corresponding Tur\'an-type existence problem. This stands in contrast with the situation for ordinary graphs, where existence and covering thresholds essentially coincide.  
Given an $r$-graph $F$, Falgas--Ravry and Zhao~\cite{FalgasRavryZhao16} introduced the notion of an $F$-covering, which is intermediate between that of the existence of a single copy of $F$ and the existence of an $F$-tiling.
 
 We say that an $r$-graph $G$ has an \emph{$F$-covering} if every vertex in $G$ is contained in a copy of $F$ in $G$.  Equivalently an $F$-covering of $G$ is a collection $C$ of copies $F$ whose union covers all of $V(G)$.
For every positive integer $i \leq r-1$, the \emph{$i$-degree $F$-covering threshold} is the function 
 \begin{align}\label{eqdef: i-degree covering threshold}
 \mathrm{c}_i(n, F):=\max \left\{\delta_i(G): \ v(G)=n, \ G \textrm{ has no }F-\textrm{covering}\right\}.
 \end{align}
 We further let the \emph{$i$-degree $F$-covering density} to be the limit\footnote{This limit can be shown to exist --- see~\cite[Footnote 1]{FalgasRavryZhao16}.} 
 \begin{align}\label{eqdef: i-degree F-covering density}
 	c_i(F):= \lim_{n\rightarrow \infty} c_i(n,F) / \binom{n-i}{r-i}.
 \end{align}

Let $K_t^{(r)}$ denote the complete $r$-graph on $t$ vertices and $K_t^{(r)-}$ denote the $r$-graph obtained by removing one edge from $K_t^{(r)}$. 
A tight $r$-uniform $t$-cycle $C_t^{(r)}$ is an $r$-graph with a cyclic ordering of its $t$ vertices such that every $r$ consecutive vertices under this ordering form an edge. 
Falgas--Ravry and Zhao~\cite{FalgasRavryZhao16} determined $c_2(F)$, where $F$ is $K_4^{(3)}$, $K_4^{(3)-}$, $K_5^{(3)-}$, and $C_5^{(3)}$. Han, Lo, and Sanhueza-Matamala \cite{2017arXiv170108115H} determined $c_{r-1}(C_t^{(r)})$ for all $r\geq 3$ and $t>2r^2$.

In this paper we investigate $c_1(n, F)$ and $c_1(F)$ for various 3-graphs $F$. We first consider $K_4^{(3)-}$.  
Let $f_n(d)$ be the function
\begin{align}\label{eqdef: f_n(d)}
f_n(d):= \binom{n-2}{2}+d -d(d-1)-\binom{d}{2}=\frac{1}{2}\left(n^2-5n+6 -3d^2 +5d \right). 
\end{align}
Observe that for fixed $n$, $f_n(d)$ is a decreasing function of  $d$ over the interval $[1, n-2]$. On the other hand $\frac{(n-1)}{2}d$ is an increasing function of $d$, so there exists a unique $d_{\star}=d_{\star}(n)$ such that $\frac{(n-1)}{2}d_{\star}=f(n,d_{\star})$, namely
\[d_{\star} =   \frac{1}{6}\left(\sqrt{13n^2 -72 n +108} -n +6 \right) = \frac{\sqrt{13}-1}{6}n +O(1).
\]
\begin{theorem}\label{theorem: k4-}
For all odd integer $n$, $\frac{n-1}2 \lfloor d_{\star}\rfloor \le c_1(n, K_4^{(3)-})\le \lfloor \frac{n-1}2 d_{\star}\rfloor$. In particular, $c_1(K_4^{(3)-})=\frac{\sqrt{13}-1}{6}=0.4342\ldots$.
\end{theorem}
The upper and lower bounds on $c_1(n,F)$ in Theorem~\ref{theorem: k4-} are apart by less than $n/2$. However it seems much more work will be needed to determine $c_1(n,F)$ exactly. As a first step in this direction, we prove the following stability theorem characterising near-extremal configurations. Let $c_{\star}= \frac{\sqrt{13}-1}{6}$.
\begin{theorem}\label{theorem: stability}
	For every $\varepsilon>0$, there exists $\delta>0$ and $n_0\in \mathbb{N}$ such that the following holds: for every $n\geq n_0$, if $H$ is a $3$-graph on $n+1$ vertices with minimum vertex degree at least $\left(c_{\star}-\delta\right)\frac{n^2}{2}$ and $x\in V(H)$ is not covered by a copy of $K_4^{(3)-}$ in $H$, then the link graph $H_x$ can be made bipartite by removing at most $\varepsilon n^2$ edges.
\end{theorem}

Next we consider $K_4^{(3)}$.  
\begin{theorem}\label{theorem: k4 bounds}
	\[\frac{19}{27}= 0.7037\ldots \leq c_1(K_4^{(3)}) \leq \frac{19}{27} + 7.4\times 10^{-9}.\]
\end{theorem}
The upper bound was derived from the flag algebra method. We believe that the lower bound is tight. As we show in Section~\ref{subsection: K4}, the problem of determining $c_1(K_4^{(3)})$ is equivalent to (a special case of) a problem about \emph{triangle-degrees} in graphs.

Given a graph $G$, the \emph{triangle-degree} of a vertex $x\in V(G)$ is the number of triangles that contains $x$. The well-studied \emph{Rademacher--Tur\'an} problem concerns the smallest \emph{average} triangle-degree among all graphs with a given edge density (the edge density $\rho(G)$ is defined as $e(G)/ \binom{v(G)}2$).  
This problem attracted significant attention (see~\cite{Bollobas76, Erdos62, Fisher89, LovaszSimonovits76, LovaszSimonovits83}) until it was resolved asymptotically by Razborov~\cite{Razborov08} using the framework of his newly-developed theory of flag algebras. Different proofs expressed in the language of weighted graphs were later found by Nikiforov~\cite{Nikiforov11} and by Reiher~\cite{Reiher16} (who generalised Razborov's result to cliques of order $4$ and of arbitrary order $t$, respectively). 

Let $t_{\mathrm{max}}(G)$ denote the maximum triangle-degree in $G$. 
(This is related to but different from the well-studied \emph{book number}, which is the maximum number of triangles containing a fixed edge of $G$, see the discussion in Section~\ref{section: concluding remarks} for details.)
For $\rho\in [0,1]$, we define 
\begin{align}\label{equation: min triangle-degree density}
\tau(\rho):=\liminf_{n \rightarrow\infty} \min \{t_{\mathrm{max}}(G)/\binom{n-1}{2}: \ v(G)=n, \ \rho(G)\geq \rho\},
\end{align}
which is the asymptotically smallest maximum scaled triangle-degree in a graph with edge density $\rho$. 
We derive the following upper bounds for $\tau(\rho)$ and conjecture that they are tight. If Conjecture~\ref{conjecture: tau value} holds, then $c_1(K_4^{(3)})=\frac{19}{27}$ (see 
Proposition~\ref{proposition: conjecture implies c1(K4) bound tight}).
\begin{theorem}\label{theorem: upper bounds on tau}
	Suppose $\rho\in [\frac{r-1}{r}, \frac{r}{r+1}]$, for some $r\in \mathbb{N}$. Then
	\[\tau(\rho)\leq \left\{ \begin{array}{ll}
	\frac{(r-1)(r-2)}{r^2} +\frac{3(r-1)}{r}\left(\rho -\frac{r-1}{r}\right) & \textrm{if }\frac{r-1}{r}\leq \rho \leq \frac{r}{r+1}-\frac{1}{3r(r+1)} \\
	\frac{r(r-1)}{(r+1)^2} -\frac{3(r-1)}{r+1}\left(\frac{r}{r+1}-\rho\right) & \textrm{if }\frac{r}{r+1}-\frac{1}{3r(r+1)}\leq \rho \leq \frac{r}{r+1}.
	\end{array}    \right. \]	
\end{theorem}
As we will see,  the constructions underpinning Theorem~\ref{theorem: upper bounds on tau} are very different from the extremal ones for the Rademacher--Tur\'an problem. 
\begin{conjecture}\label{conjecture: tau value}
	The upper bounds on $\tau(\rho)$ given in Theorem~\ref{theorem: upper bounds on tau} are tight for every $\rho\in [0,1]$.
\end{conjecture}
We use flag algebra computations to show the upper bounds from Conjecture~\ref{conjecture: tau value} are not far from optimal when $\rho \in [\frac{1}{2}, \frac{2}{3}]$ (see Theorem~\ref{theorem: flag algebra bounds}).

Following on a beautiful result of Bondy, Shen, Thomass\'e and Thomassen~\cite{BondyShenThomasseThomassen06} on a tripartite version of Mantel's theorem, Baber, Johnson and Talbot~\cite{BaberJohnsonTalbot10} 

gave a tripartite analogue of Razborov's triangle-density result. In a similar spirit, we prove Conjecture~\ref{conjecture: tau value} holds for tripartite graphs. Note that a tripartite graph on $n$ vertices can have between $0$ and $\frac{n^2}{3}$ edges. 
\begin{theorem}
\label{theorem: tripartite theorem}
Let $G$ be a tripartite graph on $n$ vertices. Then
\begin{align*}
t_{\mathrm{max}}(G)\geq \left\{ \begin{array}{ll}
	\frac{3}{2}\left(e(G)-\frac{n^2}{4}\right) & \textrm{if }  \frac{e(G)}{n^2}< \frac{3}{10},\\
	e(G)-\frac{2}{9}n^2 & \textrm{if }   \frac{3}{10}\leq \frac{e(G)}{n^2} \leq \frac{1}{3}.
\end{array}
\right.
\end{align*}
\end{theorem}	

\subsection*{Structure of the paper}
In Section~\ref{section: covering} we prove Theorems~\ref{theorem: k4-}--\ref{theorem: k4 bounds} along with bounds for $c_1(C^{(3)}_5)$ and $c_1(K^{(3)}_t)$ for $t\ge 5$. 
In Section~\ref{section:triangle-degree} we prove Theorems~\ref{theorem: upper bounds on tau} and \ref{theorem: tripartite theorem}, and give flag algebra bounds on $\tau(\rho)$. We end the paper in Section~\ref{section: concluding remarks} with a discussion of book numbers in graphs and a comparison of known results and conjectures on minimal triangle density, triangle-degree and book-number as functions of edge density. 
\subsection*{Notation}
We use standard graph and hypergraph theory notation throughout the paper. In addition, we use $[n]$ to denote the set $\{1,2, \ldots n\}$ and $S^{(r)}$ to denote the collection of all $r$-subsets of a set $S$. Where there is no risk of confusion, we identify hypergraphs with their edge-sets.
\section{Covering in $3$-graphs}\label{section: covering}

\subsection{Proof of Theorem~\ref{theorem: k4-}}
Recall that $K_4^{(3)-}$ is the (unique up to isomorphism) $3$-graph on $4$-vertices spanning $3$ edges, also known as the \emph{generalised triangle}. In this subsection, we prove Theorem~\ref{theorem: k4-}.

\begin{proof}[Proof of Theorem~\ref{theorem: k4-}]\

\noindent\textbf{Lower bound:}  let $n$ be odd, and let $d= \lfloor d_{\star}\rfloor \le (n-1)/2$. We construct a $3$-graph $H$ on $n$ vertices as follows. Set aside a vertex $v_{\star}$, and let $A\sqcup B$ be a bipartition of $V(H)\setminus \{v_{\star}\}$ into two sets of equal size. Let $G$ be an arbitrary $d$-regular bipartite graph with partition $A\sqcup B$. Now let $H$ be the $3$-graph whose $3$-edges are the union of the triples $\{v_{\star}xy: \ xy \in E(G) \}$ together with all the triples of vertices from $A\cup B$ inducing at most one edge in $G$.

Clearly, for every triple of vertices $S\subseteq A\cup B$, $S\cup\{v_{\star}\}$ induces at most two edges of $H$ and $v_{\star}$ is not contained in any copy of $K_4^{(3)-}$. Thus $c_1(n, K_4^{(3)-})\geq \delta_1(H)$. This latter quantity is easily calculated: the degree of $v_{\star}$ in $H$ is $\frac{n-1}2 d$. For any $a\in A$, there are exactly $d(d-1)$ pairs $(a', b)\in A\times B$ such that both $a'b$ and $ab$ lie in $G$, and exactly $\binom{d}{2}$ pairs $(b,b')\in B^{(2)}$ such that both $ab$ and $ab'$ lie in $G$; such pairs are the only pairs from $((A\setminus\{a\})\cup B)^{(2)}$ that do not form an edge of $H$ with $a$. In addition, there are exactly $d$ edges of $H$ containing the pair $av_{\star}$. Thus the degree of $a$ in $H$ is
\begin{align*}
\deg(a)&= \binom{n-2}{2}- d(d-1)-\binom{d}{2} + d=f_n(d).
\end{align*}
By symmetry, the degree of any vertex in $B$ is also $f_n(d)$. Thus $\delta_1(H)=\min(\frac{n-1}2 d, f_n(d))= \frac{n-1}2 d$ because $d\le d_{\star}$. Since $H$ has no $K_4^{(3)-}$-covering, it follows that $c_1(n, K_4^{(3)-})\ge \frac{n-1}2 \lfloor d_{\star}\rfloor$.

\noindent	\textbf{Upper bound:} suppose $H$ is a $3$-graph on $n$ vertices with $\delta_1(H)=\frac{n-1}2 d$ and no copy of $K_4^{(3)-}$ covering a vertex $x$ (here $n$ is not necessarily odd). We shall show that $\delta_1(H)\le \frac{n-1}2 d_{\star}$.
Note that the link graph $H_x$ of $x$ is triangle-free. Furthermore,  $v_1v_2v_3\notin E(H)$ for any triple $v_1v_2v_3$ spanning two edges in $H_x$. Let $F(v)$ denote the collection of  pairs $v_2v_3$ such that $vv_2v_3$ induces two edges in $H_x$. We know that $vv_2v_3\notin E(H)$ for every $v_2v_3\in F(v)$. Observe that $F(v)$ consists of all pairs $v_2 v_3$, where either $v_2, v_3\in \Gamma(x, v)$ or $v_2v_3\in H_x$ and exactly one of $v_2$, $v_3$ is in $\Gamma(x,v)$.

Counting non-edges of $H$ over all $v\in V\setminus\{x\}$, we thus have
	\begin{align*}
\sum_{v\in V\setminus\{x\}} \left( \binom{n-1}{2}-\deg(v)\right)
& \geq \sum_{v\in V\setminus\{x\}} n-2 -\deg(x,v)  +\vert F(v)\vert    \\
& \geq \sum_{v\in V\setminus\{x\}} \left( n-2-\deg(x,v)  + \binom{\deg(x,v)}{2}  +\sum_{v_2\in \Gamma(x,v)} (\deg(x, v_2)-1)\right)\\
	&= (n-1)(n-2) + \sum_{v\in V\setminus\{x\}}  \frac{1}{2}\left(3(\deg(x,v))^2-5\deg(x,v)\right)\\
	&\geq (n-1)(n-2)    + \frac{n-1}{2}\left(3d^2 -5d\right)= (n-1)\left( n-2+\frac{3d^2 -5d}{2}\right).
	\end{align*}
	where in the last line we used Jensen's inequality and our minimum degree assumption $\deg(x)\geq \frac{n-1}2 d$. By averaging, there exists a vertex $v\in v\in V\setminus\{x\} $ with 
\begin{align*}
\deg(v)& \leq \binom{n-1}{2}-n+2 -\frac{3d^2-5d}{2}= f_n(d).
\end{align*}	
Applying our minimum degree assumption $\deg(v)\geq \frac{n-1}2 d$ yields $\frac{n-1}2 d\leq f_n(d)$ and hence $d\leq d_{\star}$. Thus $\delta_1(H)\le \frac{n-1}2 d_{\star}$ as claimed.

\end{proof}

\subsection{Proof of Theorem~\ref{theorem: stability}}
Our proof shall make use of a consequence of Karamata's inequality.  Let $a_n\geq a_{n-1} \geq \ldots \geq  a_1$ and $b_n\geq b_{n-1} \geq \ldots \geq b_1$ be real numbers. We say that $\mathbf{a}=(a_n, \ldots, a_1)$ \emph{majorises}  $\mathbf{b}=(b_n, \ldots, b_1)$  if  $\sum_{i\geq k} a_i\geq \sum_{i\geq k} b_i$ for all $1\leq k\leq n$, with equality attained in the case $k=1$.  Karamata's inequality states that if  $\mathbf{a}$ majorises $\mathbf{b}$ and $f$ is a convex function then $\sum_i f(a_i) \geq \sum_i f(b_i)$.
\begin{lemma}\label{lemma: extended Jensen}
	Suppose $f: \ \mathbb{R}\rightarrow \mathbb{R}$ is a convex function. Let $a_1 \leq a_2\leq\ \ldots \leq  a_n$ be real numbers such that $\sum_i a_i = \bar{a} n$, and let $\eta>0$. Set $\mathcal{B}:=\{i: \ a_i \leq (1-\eta)\bar{a}\}$. Then 
	\begin{align}
	\sum_{i} f(a_i)\geq \vert \mathcal{B}\vert \cdot  f\left((1-\eta )\bar{a}\right) + \left(n-\vert \mathcal{B}\vert\right) \cdot f\left( \left(1+\frac{\eta \vert \mathcal{B}\vert}{n-\vert \mathcal{B}\vert}\right) \bar{a}\right).
	\end{align}
\end{lemma}
\begin{proof}
	Since $\eta>0$, our assumption on $\sum_i a_i$ tells us that $[n]\setminus \mathcal{B}\neq \emptyset$. If $\mathcal{B}=\emptyset$, then the claimed inequality is just Jensen's inequality. So assume $\mathcal{B}$ is nonempty and set $\vert \mathcal{B}\vert =\beta n$ for some $\beta>0$.

Let $a'_1, a'_2, \ldots , a'_n$ be given by 
\[a'_i=\left\{ \begin{array}{ll}
(1-\eta)\bar{a} & \textrm{if }i\in [\beta n]\\
\left(1+ \frac{\eta\beta}{1-\beta}\right)\bar{a} & \textrm{if }i\in [n]\setminus [\beta n].\\
\end{array}\right.\]
Observe that $\sum_i a'_i=\sum_i a_i=\bar{a}n$. Setting
\begin{align*}
x =\frac{1}{\vert \mathcal{B}\vert}\sum_{i \in \mathcal{B}}a_i \quad  \textrm{and} \quad  y =\frac{1}{n-\vert \mathcal{B}\vert}\sum_{i \in [n]\setminus\mathcal{B}}a_i,
	\end{align*}
we have 
\begin{align*}
x\leq (1-\eta)\bar{a}< \left(1+\frac{\eta \beta}{1-\beta}\right)\bar{a}\leq\frac{\bar{a} -\beta x}{1-\beta} =y.
\end{align*}
It follows readily from this that the $n$-tuple $(a_n, a_{n-1}, a_{n-2}, \ldots , a_1)$ majorises $(a'_n, a'_{n-1}, a'_{n-2}\ldots , a'_1)$. Applying Karamata's inequality to the convex function $f$ we obtain
\begin{align*}
\sum_i f(a_i)\geq  \sum_i f(a'_i)=\beta n \cdot f\left((1-\eta )\bar{a}\right) + \left(1-\beta\right) n\cdot f\left(\left(1+\frac{\eta \beta}{1-\beta}\right)\bar{a}\right).&&\qedhere
\end{align*}

\end{proof}
Another ingredient in the proof of Theorem~\ref{theorem: stability} is a classical result of Andr\'asfai, Erd{\H o}s and S\'os.
\begin{theorem}[Andr\'asfai, Erd{\H o}s, S\'os~\cite{AndrasfaiErdosSos74}]\label{proposition: Andrasfai Erdos Sos ]}
	Let $G$ be a triangle-free graph on $n$ vertices with minimum degree $\delta(G)>\frac{2n}{5}$. Then $G$ is bipartite.
\end{theorem}

\noindent With these two preparatory results in hand, the proof of Theorem~\ref{theorem: stability} is straightforward: we first use Lemma~\ref{lemma: extended Jensen} to show that the overwhelming majority of vertices in  the link graph $H_x$ have degree much larger than $\frac{2}{5}n$, whereupon we deduce from the Andr\'asfai--Erd{\H o}s--S\'os theorem that $H_x$ is almost bipartite.
\begin{proof}[Proof of Theorem~\ref{theorem: stability}]
	Recall $c_{\star}=\frac{\sqrt{13}-1}{6}=0.43\ldots >\frac{2}{5}$. Fix $\varepsilon>0$. Without loss of generality, assume that $\varepsilon <  \frac{1}{3}\left(c_{\star}-\frac{2}{5}\right)$. Pick $0< \eta < \frac{1}{3c_{\star}}\left(c_{\star}-\frac{2}{5}\right)$ and $\delta>0$ such that
	\begin{align}\label{eq: assumptions on delta}
	 \delta<\frac{1}{3}\left(c_{\star}-\frac{2}{5}\right)<\frac{c_{\star}}{6} &&\textrm{and}&&\left(\frac{1+ 6 c_{\star}}{2 c_{\star}^2 \eta^2} \right)\delta <\frac{\varepsilon}{2}	\end{align}
	both hold.

	Let $H$ be a $3$-graph with $v(H)=n+1$, $\delta_1(H)\geq \left(c_{\star}-\delta\right)\frac{n^2}{2}$. Suppose $x$ is a vertex in $H$ not covered by any copy of $K_4^{(3)-}$. Without loss of generality, assume $V(H)=[n]\cup\{x\}$. By the vertex-degree assumption, $e(H_x)=cn^2$, for some $c\geq c_{\star}-\delta$.  Let $\mathcal{B}=\left\{y\in V(H): \ \mathrm{deg}(xy)\leq c(1-\eta)n \right\}$ be the collection of vertices in $H$ whose codegree with $x$ is smaller than average by a multiplicative factor of $(1-\eta)$. Set $\vert \mathcal{B}\vert =\beta n$.

	Since $x$ is not covered by a copy of $K_4^{(3)-}$ in $H$, the following hold:
	\begin{enumerate}[(i)]
		\item $H_x$ is triangle-free;
		\item for every triple of vertices $\{y_1, y_2, y_3 \}$ inducing two edges in $H_x$, the $3$-edge $y_1y_2y_3$ is missing from $E(H)$.
	\end{enumerate}
	Property (i) implies that for every $y\in [n]$, the neighbourhood $\Gamma(xy)$ is an independent set in $H_x$, while property (ii) implies that for every $z,z' \in H_{xy}$ and every $w\in H_{xz}$, the $3$-edges $zz'y$ and $zwy$ are both missing from $E(H)$. In  particular for every $y \in [n]$, we have
	\begin{align*}
	(1-c_{\star}+\delta)\frac{n^2}{2}>\binom{n}{2}-e(H_y)\geq \binom{\vert H_{xy}\vert }{2}+\sum_{z\in H_{xy}} \left(\vert H_{xz}\vert -1\right).	
	\end{align*}
	Summing this inequality over all $y \in [n]$ and using the fact $\sum_{y\in [n]} \sum_{z\in H_{xy}} (|H_{xz}| - 1) = 2 \sum_{y\in [n]} \binom{ |H_{xy}| }{2}$, we get
	\begin{align}\label{eq: average bound on the degrees}
	(1-c_{\star}+\delta)\frac{n^3}{2}> \sum_{y\in [n]} 3\binom{\vert H_{xy}\vert}{2}.
	\end{align}
	Since the function $f(t)= \binom{t}{2}$ is convex and $\sum_{y\in [n]} \vert H_{xy}\vert =2 \vert H_x\vert =cn^2$, we can apply Lemma~\ref{lemma: extended Jensen} to bound below the right-hand side of (\ref{eq: average bound on the degrees}) by
	\begin{align}
	3 \left(\beta n \binom{c(1-\eta)n}{2} + (1-\beta )n \binom{\frac{cn^2-\beta n \cdot c(1-\eta)n} {(1-\beta)n }}{2}\right)=\frac{3c^2}{2}\left(\beta (1-\eta)^2+ \frac{\left(1-\beta(1-\eta)\right)^2}{1-\beta}  \right)n^3 + O(n^2). \notag
	\end{align}
	Inserting this inequality back into (\ref{eq: average bound on the degrees}), dividing through by $n^3$ and using $c\geq c_{\star}-\delta$ yields
	\begin{align}
		1-c_{\star}+\delta &\geq 3 (c_{\star}-\delta)^2\left(\beta (1-\eta)^2+ \frac{\left(1-\beta(1-\eta)\right)^2}{1-\beta}  \right)+O(n^{-1})\notag\\
		&= 3 (c_{\star}-\delta)^2\left(1+\eta^2 \beta + \frac{\eta^2\beta^2}{1-\beta}\right) +O(n^{-1})\\
		&>  (3c_{\star}^2-6\delta c_{\star})\left(1+\eta^2 \beta \right) +O(n^{-1}) \notag\\
		&\ge 3c_{\star}^2- 6\delta c_{\star} + 2c_{\star}^2 \eta^2 \beta + O(n^{-1}),\label{eq: simplifying}
	\end{align}
	where the last inequality holds because our choice of $\delta$ in \eqref{eq: assumptions on delta} ensures $\delta<c_{\star}/6$.
	Note that $c_{\star}$ satisfies $1-c_{\star}=3c_{\star}^2$. Rearranging terms in inequality (\ref{eq: simplifying})  gives
	\begin{align*}
	(1+ 6c_{\star}) \delta > 2c_{\star}^2 \eta^2 \beta + O(n^{-1}).
	\end{align*}
	By the second part of \eqref{eq: assumptions on delta} and the assumption that $n$ is sufficiently large, we have 
	\[
	\beta < \left(\frac{1+ 6 c_{\star}}{2 c_{\star}^2 \eta^2}\right)\delta + O(n^{-1}) <\frac{\varepsilon}{2}+O(n^{-1})< \varepsilon
	\]
	and $\vert \mathcal{B}\vert =\beta n< \varepsilon n$. Remove from $H_x$ all vertices from $\mathcal{B}$. By the definitions of $\delta, \eta, \varepsilon$, the resulting triangle-free graph $G$ has at most $n$ vertices and minimum degree at least 
	\begin{align*}
	c(1-\eta)n -\varepsilon n\geq (c_{\star}-\delta)(1-\eta)n -\varepsilon n>  \left(c_{\star}-\eta c_{\star}  -\delta -\varepsilon \right)n>  \frac{2}{5}n.
	\end{align*}
	By Theorem~\ref{proposition: Andrasfai Erdos Sos ]}, $G$ is bipartite. Since we removed only at most $\varepsilon n$ vertices from $H_x$ to obtain $G$, it follows that $H_x$ can be made bipartite by removing at most $\varepsilon n^2$ edges, as claimed. This concludes the proof of Theorem~\ref{theorem: stability}.

\end{proof}


\subsection{Proof of Theorem~\ref{theorem: k4 bounds}}
\label{subsection: K4}

Given an $r$-graph $G$, write $t_G(x)$ for the number of copies of $K_{r+1}^{(r)}$ in $G$ that cover $x$.
\begin{proposition}\label{proposition:  r-unif K(r+1) covering threshold and (r-1)-unif Kr degree}
There exists an $r$-graph $H$ on $n+1$ vertices with minimum vertex-degree $\delta_1(H)\geq \alpha \binom{n-1}{r-1}$ and no $K_{r+1}^{(r)}$-covering if and only if there exists an $(r-1)$-graph $G$ on $n$ vertices with at least $\alpha \binom{n-1}{r-1}$ edges such that for every vertex $x\in V(G)$, $t_G(x)-\deg_G(x)\leq (1-\alpha)\binom{n-1}{r-1}$.	
\end{proposition}
\begin{proof}
	In one direction, let $H$ be an $r$-graph on $n+1$ vertices with minimum degree $\alpha \binom{n-1}{r-1}$. Suppose $v_{\star}$ is not covered by any $K_{r+1}^{(r)}$ in $H$.  By the minimum degree condition on $v_{\star}$, the $(r-1)$-uniform link graph $G=H_{v_{\star}}$  contains at least $\alpha \binom{n-1}{r-1}$ edges. Also, every copy of $K_r^{(r-1)}$ in the $(r-1)$-graph $G$ must be a non-edge in the $r$-graph $H$, else together with $v_{\star}$ it would make a copy of $K_{r+1}^{(r)}$ in $H$ covering $v_{\star}$. The minimum degree condition in $H$ then implies that for every vertex $x$ in the $n$-vertex $(r-1)$-graph $G$,
	\begin{align*}
	\alpha \binom{n-1}{r-1}\leq \deg_H(x)\leq \binom{n-1}{r-1}+\deg_G(x)- t_G(x),
	\end{align*}
	implying $t_G(x)-\deg_G(x)\leq (1-\alpha)\binom{n-1}{2}$ as desired.

	In the other direction, let $G$ be an $(r-1)$-graph on $n$ vertices with at least $\alpha \binom{n-1}{r-1}$ edges such that $t_G(x)-\deg_G(x)\leq (1-\alpha)\binom{n-1}{r-1}$  for all $x\in V(G)$. We add a new vertex $v_{\star}$ to $G$ and define an $r$-graph $H$ on $V(G)\sqcup \{v_{\star}\}$ by setting the link graph of $v_{\star}$ be equal to $G$, and adding in as edges all $r$-sets from $V(G)^{(r)}$ which do not induce a copy of $K_r^{(r-1)}$ in $G$. This yields an $r$-graph on $n+1$ vertices in which $v_{\star}$ is not covered by a copy of $K_{r+1}^{(r)}$, $\deg_H(v_{\star})= e(G)\geq \alpha \binom{n-1}{r-1}$, and for every $x\in V(H)\setminus \{v_{\star}\}$, 
	\[\deg_H(x)= \binom{n-1}{r-1}-t_G(x)+\deg_G(x)\geq \alpha \binom{n-1}{r-1},\]
	so $\delta_1(H)\geq \alpha \binom{n-1}{r-1}$ as desired.
\end{proof}

\begin{corollary}\label{corollary: covering density for complete r-graphs on r+1 vertices}
		For any $r\in \mathbb{N}$, the $1$-degree covering density $c_1(K_{r+1}^{(r)})$ is the least $\alpha>0$ such that if $G$ is an $(r-1)$-graph on $n$ vertices with at least $\alpha \binom{n}{r-1}$ edges, then there is a vertex $x\in G$ contained in $t_G(x)\geq \left(1-\alpha+o(1)\right)\binom{n-1}{r-1}$ copies of $K_r^{(r-1)}$ in $G$. 
\end{corollary}

\begin{proof}[Proof of Theorem~\ref{theorem: k4 bounds}]\

	\noindent\textbf{Lower bound:}  suppose $3\vert n$ and partition $[n]$ into three sets $V_1, V_2, V_3$ of size $n/3$. Further partition each $V_i$ into two sets $V_{i,1}$ and $V_{i,2}$ of size $n/6$. Now let $G$ be the $2$-graph on $[n]$ obtained by putting in all edges of the form $V_iV_j$, with $1\leq i < j \leq 3$ and adding for each $i\in [3]$ an arbitrary $n/27$-regular bipartite graph with partition $V_{i,1}\sqcup V_{i,2}$.  An easy calculation shows $G$ is both regular and triangle-degree regular, with every vertex $x$ satisfying $\deg(x)= 19n/27$ and  $t(x)=4 n^2/27$. We have thus $t(x)-\deg(x)= \frac{8}{27}\binom{n -1}{2}+O(n)$. It follows from Proposition~\ref{proposition:  r-unif K(r+1) covering threshold and (r-1)-unif Kr degree} that there exists a $3$-graph $H$ on $n+1$ vertices with minimum degree $\left(\frac{19}{27}+O(\frac{1}{n}) \right)\binom{n-1}{2}$ and no $K_4$-covering, establishing the desired lower bound on $c_1(K_4^{(3)})$.

	\noindent \textbf{Upper bound:} set $\alpha = \frac{19}{27} + 7.4\times 10^{-9}$. By Proposition~\ref{proposition:  r-unif K(r+1) covering threshold and (r-1)-unif Kr degree}, it is enough to show that if $G$ is an $n$-vertex graph with $t_{\mathrm{max}} \leq \left(1-\alpha+o(1)\right)\binom{n-1}{2}$, then $e(G)\leq \left(\alpha +o(1)\right)\binom{n-1}{2}$. This is done in Proposition~\ref{prop: pointwise flag algebra bounds} in the next section via a simple flag algebra calculation. 
\end{proof}

\subsection{$C_5^{(3)}$}\label{subsection: c5}

\begin{theorem}\label{theorem: c5 bounds}
	$0.55\ldots=\frac{5}{9}\leq c_1(C_5^{(3)})\leq 2-\sqrt{2}=0.58\ldots$
\end{theorem}
\begin{proof} 
	
	\noindent \textbf{Lower bound: } we construct a $3$-graph on $[3n+1]$ as follows. Set aside $v_{\star}=3n+1$, and partition the remaining vertices into an $n$-set $A$ and a $2n$-set $B$. Let $H$ be the $3$-graph on $[3n+1]$ obtained by setting the link graph of $v_{\star}$ to be the union of a clique on $A$ and a clique on $B$, and adding all triples of the form $AAB$ or $ABB$. Every path of length $3$ in the link graph of $v_{\star}$ gives rise to an independent set in $H$, hence there is no copy of the strong $5$-cycle $C_5$ covering $v_{\star}$ in $H$. The degree of $v_{\star}$ in $H$ is $\binom{n}{2}+\binom{2n}{2}= \frac{5}{9}\binom{3n}{2}-\frac{2n}{3}$, and the degree in the rest of the graph are all at least
	\begin{align*}
	\min\left( \left(\vert A\vert -1\right)\vert B\vert + \binom{\vert B\vert }{2}, \left(\vert B\vert -1\right)\vert A\vert + \binom{\vert A\vert }{2}\right)= n(2n-1)+\binom{n}{2}= \frac{5}{9}\binom{3n}{2}-\frac{2n}{3}. 
	\end{align*}	
	Thus $c_1(3n+1, C_5)\geq \frac{5}{9}\binom{3n}{2}-\frac{2n}{3}$, as desired.

\noindent\textbf{Upper bound: } Mubayi and R\"odl \cite[Theorem 1.9]{MubayiRodl02} proved that $\pi(C_5^{(3)})\leq 2-\sqrt{2}$. An easy modification of their proof shows that $\alpha=2-\sqrt{2}$ is in fact an upper bound for the covering threshold. Indeed, let $H$ be a $3$ graph on $n$ vertices with $\delta_1(H)\geq \alpha \binom{n-1}{2}+c(2n-1)$, for some $c\geq 10$. Let $x$ be an arbitrary vertex in $V(H)$. By averaging, there exists $y\in V(H)$ such that $\deg(xy)\geq\alpha n$.  Form the multigraph $G=H_x\cup H_y$ as in ~\cite[ Proof of Theorem 1.9, p 151]{MubayiRodl02}. Then \cite[Claim, p 151]{MubayiRodl02} shows that if there is no copy of $C_5^{(3)}$ covering the pair $xy$, then $G$ satisfies the conditions of \cite[Lemma 6.2, p 149]{MubayiRodl02}, and one can conclude as Mubayi and R\"odl do that one of $x$ and $y$ has degree at most $\alpha \binom{n-1}{2} + c(2n-1) -n$ in $H$ , contradicting our minimum degree assumption.
	
\end{proof}

\subsection{$K_t^{(3)}$, $t\geq 5$}\label{subsection: kt covering tiling}

\begin{proposition}\label{theorem: inductive upper bound on covering density}
	For all $t\geq 4$, $c_1(K_{t+1}^{(3)})\leq \frac{-1+\sqrt{3-2c_1(K_t^{(3)} )} }{1-c_1(K_t^{(3)})}$.
\end{proposition}
\begin{proof} Let $\varepsilon>0$ and $n$ be sufficiently large. Suppose that $H$ is a $3$-graph on $n$ vertices with $\delta_1(H)\geq \alpha \frac{n^2}{2}$ for some $\alpha>0$ satisfying $1+ \frac2{\alpha} - \frac2{\alpha^2}= c_1(K_t^{(3)})+ \varepsilon$.  
Let $v_{\star}$ be an arbitrary vertex. By averaging, there exists a vertex $x\in V\setminus \{v_{\star}\}$ and an $\alpha n$-set $V'$ such that $V'\subseteq \Gamma(x,v_{\star})$.
Observe that 
\[
e(H_x[V'])\ge e(H_x) - |V'|(n - |V'|) - \binom{n- |V'|}2
\]
and an analogous bound holds for $e(H_{v_{\star}}[V'])$. Thus
\begin{align}
e( {H}_x[V']\cap {H}_{v_{\star}}[V'])& \geq  e(H_x) + e(H_{v_{\star}})- 2\vert V'\vert (n-\vert V'\vert) - 2\binom{n- |V'|}2 -\binom{\vert V'\vert }{2} \notag\\
&\geq \left(\alpha^2+ 2\alpha -2\right)\frac{n^2}2 + O(n).
\label{equation: bound on delta1 H' in Kt upper bound part 1}
\end{align} 
On the other hand, for any $y\in V'$, we have
\begin{align}
\vert \Gamma(y)\cap (V'\cup\{x,v_{\star} \})^{(2)}\vert & \geq \deg(y)- (\vert V'\vert +1)(n-\vert V'\vert -2) - \binom{n-\vert V'\vert -2}{2}\notag \\
& \geq \left({\alpha}^2 + \alpha - 1 \right) \frac{n^2}2 + O(n).
\label{equation: bound on delta1 H' in Kt upper bound part 2}
\end{align}

Note that
\begin{align}\label{eq:min2a}
\min\left( \alpha^2 + 2\alpha -2, {\alpha}^2 + \alpha - 1 \right)= \alpha^2 + 2\alpha -2 = (c_1(K_t^{(3)})+\varepsilon) \alpha^2.
\end{align}
Let $H'$ be the $3$-graph obtained by taking  $H[V']$ and adding a new vertex $z$ whose link graph consists precisely of those pairs $yy' \in E({H}_x[V']\cap {H}_{v_{\star}}[V'])$. By \eqref{equation: bound on delta1 H' in Kt upper bound part 1}, \eqref{equation: bound on delta1 H' in Kt upper bound part 2} and \eqref{eq:min2a}, $\delta_1(H')\geq  \left(c_1(K_t)+\frac{\varepsilon}2 \right)\binom{v(H')}{2}$. Thus provided $\alpha n = v(H')$ is sufficiently large, there must be a set $S\subseteq V'$ such that $S\cup\{z\}$ induces a copy of $K_{t}^{(3)}$ in $H'$ covering $z$. But then by construction of $H'$, this implies that $S\cup\{x, v_{\star}\}$ induces a copy of $K_{t+1}^{(3)}$ covering $v_{\star}$ in $H$. It follows that $\alpha \ge c_1(K_{t+1}^{(3)})$, and hence (since $\varepsilon>0$ was arbitrary) that $c_1(K_{t+1}^{(3)})\leq \frac{-1+\sqrt{3-2c_1(K_t^{(3)} )} }{1-c_1(K_t^{(3)})}$.
\end{proof}

\begin{proposition}\label{proposition: construction of lower bounds for covering density of Kt}
Suppose there exists a $3$-graph $H$ on $[N]$ such that
\begin{enumerate}[(i)]
	\item every vertex of $H$ has degree at most $d$;
	\item every $t$-set of vertices from $V(H)$ spans at least one edge.
\end{enumerate}
Then we have
\[c_1\left(K_{t+1}^{(3)}\right)\geq \min\left(1-\frac{1}{N}, 1-\frac{2d}{N^2}\right)\]
\end{proposition}
\begin{proof}
We construct a $3$-graph $G$ on $[Nn+1]$ as follows. Set aside $v_{\star}=Nn+1$, and partition the remaining vertices into $n$-sets $V_1, V_2, \ldots V_N$. Now let the link graph of $v_{\star}$ in $G$ be the complete $N$-partite graph on $[Nn]$ with partition $\sqcup_{i=1}^N V_i$. To make up the remainder of the edges of $G$, add in all triples $v_1v_2v_3$ from $[Nn]^{(3)}$ with $v_j \in V_{i_j}$ for $j=1,2,3$ and $i_1i_2i_3 \notin E(H)$.

Clearly $\deg_G(v_{\star})= \binom{nN}{2}-N\binom{n}{2}=\left(1-\frac{1}{N}\right)\binom{nN}{2}+O(n)$, and every other vertex $x\in [nN]$ with $x\in V_i$ has degree
\[\deg_G(x)= n(N-1)+ \binom{nN-1}{2} -\deg_H(i) n^2\geq \left(1- \frac{2d}{N^2}\right)\binom{nN}{2}+O(n).\]
Thus $\delta_1(G)\geq \min 	\left(1-\frac{1}{N}, 1-\frac{2d}{N^2}\right)\binom{nN}{2}+O(n)$. Furthermore, every complete graph $G_{v_{\star}}[T]$ on $\vert T\vert = t$ vertices in the link graph of $v_{\star}$  in $G$ meets $t$ different parts $V_{i_1}, \ldots , V_{i_t}$ from our partition of $[nN]$. By assumption, $i_1i_2, \ldots i_t$ spans at least one edge of $H$, whence we have that at least one of the triples from $T^{(3)}$ is missing from $E(G)$. In particular $\{v_{\star}\}\cup T$ does not span a copy of $K_{t+1}^{(3)}$ in $G$, and $G$ fails to have a $K_{t+1}^{(3)}$-cover. The proposition follows.
	\end{proof}

A natural family of $3$-graphs for applications of Proposition~\ref{proposition: construction of lower bounds for covering density of Kt} are \emph{Steiner triple systems} (STS), where each pair of vertices is contained in a unique edge. Let $\alpha_t$ denote the minimum of the independence number over all STS of order $t$. The unique (up to isomorphism) STS of orders $3$ and $7$ are the $3$-edge $K_3^{(3)}$ and the Fano plane $S_7$ respectively, which give $\alpha_3=2$, $\alpha_7=4$. The affine plane of order $9$, $S_9$, is the unique up to isomorphism STS of order $9$ and has $\alpha(S_9)=\alpha_9=4$. It is further known that $\alpha_{13}=6$, $\alpha_{15}=6$~\cite{MathonPhelpsRosa82}, and $\alpha_{19}=7$~\cite{ColbournForbesGrannellGriggsKaskiOstergardPikePottonen10} (see also the monograph of Kaski and \"Osterg{\aa}rd~\cite{KaskiOstergaard06}).

\begin{proposition}\label{prop: kt bounds}
	\begin{align*}
	0.8888\ldots =\frac{8}{9} &\leq c_1\left(K_6^{(3)}\right)\leq 0.947962\ldots\\
	0.9333\ldots =\frac{14}{15} &\leq c_1\left(K_{8}^{(3)}\right)\leq  0.98793\ldots\\ 
	0.9473\ldots =\frac{18}{19}&\leq c_1\left(K_{9}^{(3)}\right) \leq 0.99404\ldots  \ .
	\end{align*}
\end{proposition}
\begin{proof}\ \newline
\noindent\textbf{Lower bound:}  apply Proposition~\ref{proposition: construction of lower bounds for covering density of Kt} to STS of orders $9$, $15$ and $19$ with minimum independence numbers, and observe that an STS of order $t$ is a $d=\frac{t-1}{2}$-regular $3$-graph, so that $\min(1-\frac{1}{t}, 1-\frac{2d}{t^2})=1-\frac{1}{t}$.

 	\noindent\textbf{Upper bound:} repeatedly apply Proposition~\ref{theorem: inductive upper bound on covering density} with our upper bound $c_1(K_4)\leq \frac{19}{27} + 7.4\times 10^{-9}$ from Theorem~\ref{theorem: k4 bounds}.
\end{proof}
\begin{remark}
	The lower bounds on the covering densities in Proposition~\ref{prop: kt bounds} above are strictly stronger than the bounds one gets from the conjectured values of the corresponding Tur\'an densities. 

In each case, they are about $5\times 10^{-2}$ below our upper bounds. Note that if one applies Proposition~\ref{proposition: construction of lower bounds for covering density of Kt} to the unique STS on $3$-vertices, one gets  a lower bound of $2/3$ for $c_1(K_4^{(3)})$. We obtained an improvement of this bound in Theorem~\ref{theorem: k4 bounds} by almost $5\times 10^{-2}$ by adding a few edges in the link graph of $v_{\star}$ and deleting a few triples meetings the corresponding pairs. It seems natural to believe a similar (albeit significantly more intricate) process would similarly improve the lower bounds in Proposition~\ref{prop: kt bounds}. If we had to guess, we would thus say that the true value of $c_1(K_t^{(3)})$ for $t=6,8,9$ probably lies closer to the upper bounds we give.
\end{remark}

\noindent For completeness, we give (very weak) bounds on $c_1\left(K_5^{(3)}\right)$, which show $ c_1\left(K_4^{(3)}\right)< c_1\left(K_5^{(3)}\right)<c_1\left(K_6^{(3)}\right)$.
\begin{proposition}\label{proposition: k5 bounds}
	$\frac{3}{4}\leq c_1\left(K_5^{(3)}\right)\leq 0.8842\ldots$
\end{proposition}

\begin{proof}\ \newline
	\noindent\textbf{Lower bound:}  consider a partition of $[2n]$ into $n$-sets, $[2n]=V_1\sqcup V_2$. Let $G$ be the $3$-graph on $[2n]$ whose edge-set consists of all triples meeting both $V_1$ and $V_2$. It is easily checked that $G$ is $K_5^{(3)}$-free and has minimum degree $\binom{2n-1}{2}-\binom{n-1}{2}=\frac{3}{4}\binom{2n-1}{2}+O(n)$, giving us the required lower bound.

	\noindent\textbf{Upper bound:} apply Proposition~\ref{theorem: inductive upper bound on covering density} with our upper bound $c_1(K_4^{(3)})\leq \frac{19}{27} + 7.4\times 10^{-9}$ from Theorem~\ref{theorem: k4 bounds}.	
\end{proof}

\section{Triangle-degree in graphs} \label{section: triangle-degree}
\label{section:triangle-degree}
In this section, we investigate the problem of minimising the maximum triangle-degree $\tau (\rho)n^2/2$ in a $2$-graph with a given edge density $\rho$. We give upper bound constructions for $\tau(\rho)$, which we conjecture are best possible. We show our conjecture holds for tripartite graphs and use flag algebra computations to bound below $\tau(\rho)$ for general graphs with $1/2<\rho \leq 2/3$.


\subsection{Proof of Theorem~\ref{theorem: upper bounds on tau}}

\begin{proposition}\label{proposition: conjecture implies c1(K4) bound tight}
Conjecture~\ref{conjecture: tau value} implies $c_1(K_4^{(3)})=\frac{19}{27}$.
\end{proposition}
\begin{proof}
Suppose $\rho= c_1(K_4^{(3)})$. 
By Proposition~\ref{proposition:  r-unif K(r+1) covering threshold and (r-1)-unif Kr degree}, there exist a sequence $(G_n)_{n\in \mathbb{N}}$ of $2$-graphs with $v(G_n)\rightarrow \infty$, $\rho(G_n)\geq \rho+o(1)$ and $t_{\mathrm{max}}(G_n)\leq (1-\rho +o(1))\binom{v(G_n)-1}{2}$. In particular, this implies that $\tau(\rho)\leq 1-\rho$. 	If Conjecture~\ref{conjecture: tau value} is true, then since $\frac{19}{27}\in (\frac23, \frac34)$, we have $\tau(\frac{19}{27})= \frac{8}{27}$ and $\tau(\frac{19}{27}+\eps)> \frac{8}{27}$ for sufficiently small $\eps>0$. Hence $\rho \leq \frac{19}{27}$. Together with the lower bound from Theorem~\ref{theorem: k4 bounds}, we conclude that $c_1(K_4^{(3)})= \frac{19}{27}$.
\end{proof}

We now give constructions for two families of graphs used in the proof of Theorem~\ref{theorem: upper bounds on tau}.

\begin{construction}[Lower interval construction]\label{construction: upwards}
	Let $\rho \in [\frac{r-1}{r}, \frac{r}{r+1}- \frac{1}{3r(r+1)}]$ for some $r\in \mathbb{N}$. 
Suppose $n\in \mathbb{N}$ is divisible by $2r$.  	
Consider a balanced complete $r$-partite graph on $[n]$ with parts $V_1, \ldots V_r$. Add inside each $V_i$ an arbitrary $d$-regular triangle-free graph $H_i$, where $d=\left\lfloor \left(\rho-\frac{r-1}{r}\right)n\right\rfloor $.	Such triangle-free graphs exist since $d\leq \frac{2}{3(r+1)}\frac{n}{r}$ (by our upper bound on $\rho$), which is less than $n/(2r)$ (so one could take $H_i$ to b a balanced bipartite graph, for example). The resulting graph is $\lfloor\rho n\rfloor $-regular. We denote by $\mathcal{G}^u_{\rho,n}$ the family of all graphs that can be constructed in this way. 
	\end{construction}
\begin{construction}[Upper interval construction]\label{construction: downwards}
	Let $\rho \in [\frac{r}{r+1}- \frac{1}{3r(r+1)}, \frac{r}{r+1}]$ for some $r\in \mathbb{N}$. 
Suppose $n\in \mathbb{N}$ is divisible by $2(r+1)$.  	
Consider a balanced complete $(r+1)$-partite graph on $[n]$ with parts $V_1, \ldots V_{r+1}$.  Equally divide each $V_i$ into $V_i'$ and $V_i''$. Let $\phi: \ [r+1]\rightarrow [r+1]$ be any bijection with the property that $\phi(i)\neq i$ for all $i \in [r+1]$ (any permutation of $[r+1]$ with no fixed point will do). Now for every $i \in [r+1]$, replace the complete bipartite graph between $V_i'$ and $V_{\phi(i)}''$ by an arbitrary $d$-regular bipartite subgraph $H_i$, where $d =\left\lceil \left(\rho-\frac{r}{r+1}+\frac{1}{2(r+1)}\right)n\right\rceil$. The resulting graph is $\lceil \rho n\rceil$-regular. We denote by $\mathcal{G}^d_{\rho,n}$ the family of all graphs that can be constructed in this way. 

\end{construction}
\begin{remark}
	The choices of the graphs $H_i$ in both Construction~\ref{construction: upwards} and~\ref{construction: downwards} give rise to very different graphs (lying at edit distance $\Omega(n^2)$  from each other). In particular if Conjecture~\ref{conjecture: tau value} is correct, then the problem of minimising the maximum triangle-degree is not stable. This stands in some contrast with the Rademacher--Tur\'an problem for triangles, for which Pikhurko and Razborov \cite{PikhurkoRazborov16} obtained a stability result, establishing that there is an asymptotically unique way of minimising the number of triangles for a given edge-density.   This instability is observed even at the level of subgraph frequencies, as e.g. in the first construction we could take as $H_i$  a subgraph of a blow-up of the five-cycle instead of a bipartite graph, provided $\rho \leq \frac{r-1}{r}+\frac{2}{5r}$.
	
In particular, this suggests Conjecture~\ref{conjecture: tau value} may be harder to resolve  than the Rademacher--Tur\'an problem for triangles, and might not amenable to standard flag algebraic approaches due to the instability of the extremal examples.
	
\end{remark}

\begin{proof}[Proof of Theorem~\ref{theorem: upper bounds on tau}]
Assume that $\rho \in [\frac{r-1}{r}, \frac{r}{r+1}]$ for some $r\in \mathbb{N}$. When $r=1$, a $\rho n$-regular bipartite graph on $n$ vertices (we may use $\mathcal{G}^u_{\rho, n}$ and $\mathcal{G}^d_{\rho, n}$ as well)
shows that $\tau(\rho)=0$. So we may assume that  $r\geq2$.
	
First assume that $\rho \in [\frac{r-1}{r}, \frac{r}{r+1}- \frac{1}{3r(r+1)}]$. Consider an arbitrary graph $G$ of $\mathcal{G}^u_{\rho, n}$, for some $n$ divisible by $2r$. Pick a vertex $x\in V_i$. Let us compute the triangle-degree of $x$. There are at most $(n-\vert V_i\vert) d$ pairs $(y,x')$ with $x'\in V_{i}\setminus \{x\}$, $y\in [n]\setminus V_i$ and $xx'y$ forming a triangle in $G$. Further, there are at most $\frac{1}{2}\sum_{j\neq i} \vert V_j\vert d$ pairs $(y,y')$ with $y,y'\in V_j \neq V_i$ and $xyy'$ forming a triangle in $G$. Finally, there are at most $\frac{1}{2}\sum_{j: \ j\neq i} \sum_{k: \ k \neq i,j}\vert V_j\vert \vert V_k\vert$ pairs $(y,z)$ with $y\in V_j$, $z\in V_k$, $V_i, V_j, V_k$ all distinct and $xyz$ forming a triangle in $G$. Since each part $V_i$ is triangle-free by construction, there are no other triangles in $G$ containing $x$, and the triangle-degree of $x$ is thus at most
	\begin{align*} 
	t_G(x)&= \frac{r-1}{r}n\left\lfloor\left(\rho-\frac{r-1}{r}\right)n\right\rfloor+\frac{r-1}{2r}n \left\lfloor\left(\rho-\frac{r-1}{r}\right)n\right\rfloor+ \frac{(r-1)(r-2)}{2r^2} {n^2} \\
	&= \left(\frac{(r-1)(r-2)}{r^2}+ \frac{3(r-1)}{r} \left(\rho-\frac{r-1}{r}\right) \right)\frac{n^2}{2}+O(n).
	\end{align*}
	This gives the claimed upper bound on $\tau(\rho)$ for $\rho \in [\frac{r-1}{r}, \frac{r}{r+1}-\frac{1}{3r(r+1)}]$.

	Next, assume that $\rho \in [ \frac{r}{r+1}- \frac{1}{3r(r+1)}, \frac{r}{r+1}]$. Consider an arbitrary graph $G$ of $\mathcal{G}^d_{\rho, n}$, for some $n$ divisible by $2(r+1)$. Pick a vertex $x\in V_i'$ (the case when $x\in V''_i$ is analogous) . When computing the triangle-degree of $x$, it is  more convenient to count the number of triangles containing $x$ in the balanced complete $(r+1)$-partite graph from which an edge was deleted when constructing $G$. Observe that 
every triangle has lost at most one edge.

	First of all, we have lost $(\vert V_{\phi(i)}''\vert -d)\left(\frac{r-1}{r+1}\right)n$
	triangles of the form $xyz$ with $y\in V_{\phi(i)}''$. Secondly, for every $y\in [n]\setminus \left(V_i\cup V''_{\phi(i)}\cup V'_{\phi^{-1}(i)}\right)$, there are $ \frac{n}{2(r+1)}  -d $ vertices $z\in [n]\setminus \left(V_i\cup V''_{\phi(i)}\cup V'_{\phi^{-1}(i)}\right)$ such that the edge $yz$ was lost. This results in $\frac{r-1}{r+1} \frac{n}2 (\frac{n}{2(r+1)} -d)$ lost triangles $xyz$. In total there are 
	\begin{align*} 
	\left(\frac{n}{2(r+1)} -d \right)\left(\frac{r-1}{r+1}\right)n + \frac{n}{2 }\left(\frac{r-1}{r+1}\right)\left(\frac{n}{2(r+1)}  -d\right)=3\left(\frac{r-1}{r+1}\right) \left(\frac{r}{r+1}-\rho\right)\frac{n^2}{2}+O(n) 
	\end{align*}
	lost triangles for $x$. Subtracting this quantity from the triangle-degree of $x$ in the original  complete balanced $(r+1)$-partite graph, we get
	\[t_G(x)= \left( \frac{r(r-1)}{(r+1)^2} - 3\left(\frac{r-1}{r+1} \right)\left(\frac{r}{r+1}-\rho\right)\right) \frac{n^2}{2}+O(n).\]
	This gives the claimed upper bound on $\tau(\rho)$  for $\rho \in [ \frac{r}{r+1}- \frac{1}{3r(r+1)}, \frac{r}{r+1}]$.
\end{proof}

\subsection{Proof of Theorem~\ref{theorem: tripartite theorem}}
\label{subsection: tripartite graphs}
For this range of $e(G)$, Conjecture~\ref{conjecture: tau value} states that for any $n$-vertex graph $G$,
\begin{align*}
t_{\mathrm{max}}(G)\geq \left\{ \begin{array}{ll}
	0 & \textrm{if } e(G)\leq \frac{n^2}{4},\\
	\frac{3}{2}\left(e(G)-\frac{n^2}{4}\right) +O(n) & \textrm{if }   \frac{n^2}{4}\leq e(G)\leq \frac{11}{36}n^2,\\
	e(G)-\frac{2}{9}n^2 +O(n) & \textrm{if }   \frac{11}{36}n^2\leq e(G)\leq \frac{1}{3}n^2.
\end{array}
\right.
\end{align*}
		
\begin{remark}\label{remark: tripartite theorem implies tripartite conjecture}
	Since $\frac{3}{10}<\frac{11}{36}$ and since for $e(G)<\frac{11}{36}n^2$ we have
	\[e(G)-\frac{2}{9}n^2 > \frac{3}{2}\left(e(G)-\frac{n^2}{4}\right),\]
	Theorem~\ref{theorem: tripartite theorem} implies that Conjecture~\ref{conjecture: tau value} holds true for all tripartite graphs.
\end{remark}
\begin{proof}[Proof of Theorem~\ref{theorem: tripartite theorem}]
Let $G$ be an $n$-vertex tripartite graph  with partition $A\sqcup B\sqcup C$. 
Since $t_{\mathrm{max}}(G)$ is nonnegative, «
we only need to consider the case when $e(G)>\frac{n^2}{4}$. Assume without loss of generality that 
\[|A|\ge |B|\ge |C|. \]
Suppose $\vert A\vert =xn$ and $ \vert B\vert = yn$ (and so $|C|= (1-x-y)n$).
Then $x \geq y \geq \frac{1-x}{2}\geq 0$, and in particular $x\geq \frac{1}{3}$. Since $|B| |C|\le (\frac{1-x}2)^2$,
we have
\[ 
e(G)\leq \vert A\vert (n-\vert A\vert ) + \vert B\vert \vert C\vert\leq 
\left(x(1-x)+ \left(\frac{1-x}{2}\right)^2\right)n^2.\]
The function of $x$ on the right-hand side has derivative $\frac{3}{2}\left(\frac{1}{3}-x\right)n^2\leq 0$ for $x\geq \frac{1}{3}$, and attains the value $\frac{n^2}{4}$ at $x=\frac{2}{3}$. Since $e(G)> n^2/4$, 
we must have $x<\frac{2}{3}$.

Write $\alpha$ for the edge density of $G$ between parts $B$ and $C$, $\beta $ for the edge density between parts $A$ and $C$, and $\gamma$ for the edge density between parts $A$ and $B$. So we have
\begin{align*}
\frac{e(G)}{n^2}= \gamma xy + \beta x (1-x-y) + \alpha y (1-x-y).
\end{align*}
Since $x\geq y \geq 1-x-y$, if $\alpha+\beta +\gamma=S\leq  2$ then $e(G)/n^2$ is maximised by letting $\gamma =\min(S,1)$, $\beta=S- \gamma$, and $\alpha=0$, i.e. by making $G$ bipartite. But a bipartite graph contains at most $\frac{n^2}{4}$ edges, contradicting our lower bound on $e(G)$. Thus we assume $\alpha +\beta +\gamma=2+s$ for some $s$ with $0<s \leq 1$. Further, if $x,s$ are fixed with $x\geq y \geq 1-x-y$, then $e(G)/n^2$ is maximised by letting $\gamma =1$, $\beta=1$, $\alpha=s$ and $y=\frac{1-x}{2}$. 
In other words, we have
\begin{align} \label{inequality: upper bound on edge density of tripartite graph}
\frac{e(G)}{n^2}\leq  f_1(x,s):= x-x^2 +\frac{s}{4}(1-x)^2.
\end{align}
Since
 \begin{align*}
 \frac{\partial}{\partial x} f_1(x,s)&= 1- 2x -\frac{s}{2}(1-x) = \left(\frac{2-s}{2}\right) - \left(\frac{4-s}{2}\right)x,
 \end{align*}
when $s$ is fixed, $f_1(x,s)$ attains a maximum at $x_{\star}= \frac{2-s}{4-s}\in [\frac{1}{3}, \frac{1}{2}]$ (as $0\le s\le 1$). Consequently,  
\begin{align}\label{eq:eGf1}
\frac{e(G)}{n^2} \leq f_1(x,s)&\leq f_1(x_{\star}, s) =  \frac{1}{4-s}.
\end{align}

On the other hand, we can give a lower bound on $t_{\mathrm{max}}(G)/n^2$ as follows. Select vertices $a\in A$, $b\in B$ and $c\in C$ uniformly at random. By the union bound,
\begin{align*} 
\mathbb{P}(abc \textrm{ induces a triangle}) \geq \mathbb{P}(ab \in E(G)) - \mathbb{P}(bc\notin E(G))-\mathbb{P}(ac \notin E(G))= \alpha +\beta+\gamma-2=s.
\end{align*}
In particular, $G$ must contain at least $sxy(1-x-y)n^3$ triangles. By averaging over all vertices $c\in C$ we have 
\begin{align*}
\frac{t_{\mathrm{max}}(G)}{n^2}\geq  \frac{sxy(1-x-y) n^3}{\vert C\vert n^2} = sxy.
\end{align*}
Since $x\geq y \geq {1-x-y}$, for fixed $s$ and $x$, $sxy$ is minimised by setting $y=\frac{1-x}{2}$. Thus 
\begin{align} \label{inequality: lower bound on triangle-degree in tripartite graphs}
\frac{t_{\mathrm{max}}(G)}{n^2}\geq  f_2(x,s):= \frac{sx(1-x)}{2}.
\end{align}
Having done these preparatory work, we can now prove the theorem by using the following claim. 
\begin{claim}\label{clm:ts}
\begin{align*}
t_{\mathrm{max}}(G)\geq \left\{ \begin{array}{ll}
	e(G)-\frac{2}{9}n^2 & \textrm{if }   s\ge \frac{2}{3},\\
	\frac{3}{2}\left(e(G)-\frac{n^2}{4}\right) & \textrm{if } s< \frac{2}{3}.
\end{array}
\right.
\end{align*}
\end{claim}
To see why Claim~\ref{clm:ts} implies Theorem~\ref{theorem: tripartite theorem}, first assume $e(G)\ge \frac3{10} n^2$. By \eqref{eq:eGf1}, we have $s\ge 2/3$. Then Claim~\ref{clm:ts} gives that $t_{\mathrm{max}}(G)\geq e(G)-\frac{2}{9}n^2$. Now assume $e(G)< \frac3{10} n^2$. If we still have $s\ge 2/3$, then by Claim~\ref{clm:ts},
\[
t_{\mathrm{max}}(G)\geq e(G)-\frac{2}{9}n^2 >  \frac{3}{2}\left(e(G)-\frac{n^2}{4}\right)
\]
because $e(G)< \frac{11}{36} n^2$. Otherwise $s< 2/3$ and Claim~\ref{clm:ts} implies that $t_{\mathrm{max}}(G)\ge \frac{3}{2}\left(e(G)-\frac{n^2}{4}\right)$, as desired.
\end{proof}

\begin{proof}[Proof of Claim~\ref{clm:ts}]
{\ }

\noindent \textbf{Case 1: $s\geq \frac{2}{3}$.}  By inequalities (\ref{inequality: upper bound on edge density of tripartite graph}) and (\ref{inequality: lower bound on triangle-degree in tripartite graphs}), we have
\begin{align*} 
	\frac{e(G)}{n^2}-\frac{t_{\mathrm{max}}(G)}{n^2} & \leq f_1(x,s) - f_2(x,s)= x-x^2+\frac{s}{4}(1-x)^2 - \frac{s}{2}x(1-x)
\end{align*}

It is an easy exercise in calculus to show that as a function of $x\in (0,1)$, the right-hand side is maximized at $x_{\star}=\frac{2-2s}{4-3s}\le \frac 13$ (as $s\geq \frac{2}{3}$), and is decreasing in $[x_{\star}, +\infty)$. Under our assumption $x\ge 1/3$, we thus have
\begin{align*}
f_1(x,s) - f_2(x,s) \le f_1\left(\tfrac{1}{3},s\right) - f_2\left(\tfrac{1}{3},s\right)= \frac{2}{9}.
\end{align*}
This implies that $t_{\mathrm{max}}(G)\geq e(G)-\frac{2}{9}n^2$.

\noindent \textbf{Case 2: $0<s < \frac{2}{3}$.} 
By inequalities (\ref{inequality: upper bound on edge density of tripartite graph}) and (\ref{inequality: lower bound on triangle-degree in tripartite graphs}) we have
\begin{align}\label{inequality: 1.5 e -t bound, case s <2/3}
\frac{3}{2}\frac{e(G)}{n^2} -\frac{3}{8}- \frac{t_{\mathrm{max}}(G)}{n^2}&\leq \frac{3}{2}f_1(x,s) -\frac{3}{8} - f_2(x,s)= \frac{3}{2}\left( x(1-x)-\frac{1}{4}\right) + \frac{s}{8}(1-x)\left(3-7x \right).
\end{align}
If $x\in [\frac{3}{7},1]$, then both terms on the right-hand side are non-positive. Assume now that $x \in [\frac{1}{3}, \frac{3}{7})$. Then for such values of $x$, the right-hand side is an increasing function of $s$.  Applying our assumption on $s$, its value is at most
\begin{align*}
\frac{3}{2}f_1(x,\tfrac{2}{3}) -\frac{3}{8} - f_2(x,\tfrac{2}{3})&=-\frac{1}{8} + \tfrac{2}{3}x-\frac{11}{12}x^2.
\end{align*}
The discriminant of this quadratic is $\frac{4}{9} -4\cdot \frac{1}{8}\cdot\frac{11}{12}=-\frac{1}{72}<0$, so the expression above is (strictly) non-positive. We deduce that the right-hand side of (\ref{inequality: 1.5 e -t bound, case s <2/3}) is non-positive for every value of $x\in [0,1]$.  This yields $t_{\mathrm{max}}(G)\geq \frac{3}{2} \left(e(G)-\frac{1}{4}n^2\right)$.
\end{proof}

\subsection{Flag algebra bounds}\label{subsection: flag algebra bounds}
In this section we will employ  Razborov's~\cite{Razborov07} flag algebra framework, and more specifically his semidefinite method, to obtain bounds for some of the problems we study.  The semi-definite method has become a fairly standard tool in extremal combinatorics --- see e.g.~\cite{Razborov13} for a survey of some of the early applications.  As the method is well established and we have only obtained non-sharp bounds using it, we give only minimal details here, without expounding on the underlying theoretical machinery.

We have used \emph{Flagmatic} to perform our flag algebra computations; this is an open source program written by  Emil Vaughan, and later developed further by Jakub Sliacan~\cite{Sliacan}, who currently maintains a Flagmatic page on GitHub~\cite{Sliacan}.   We have used Vaughan's Flagmatic 2.0 in this paper. We refer the reader to~\cite{FalgasRavryVaughan13} and to the Flagmatic 2.0 section on the webpage~\cite{Sliacan}  for a description of the inner workings of \emph{Flagmatic} and download links for the program.  Our calculations involve the use of flag inequalities given as `axioms'. The use of such `axioms' first appeared in~\cite{FalgasRavryMarchantPikhurkoVaughan15}, where an edge-maximisation problem was solved subject to a codegree constraint. We refer a reader interested in the details to either Section~3 in that paper or to the Flagmatic 2.0 webpage~\cite{Sliacan}.

Let $T_{1}$ denote the $([1], \emptyset)$-flag consisting of a triangle with one vertex labelled $1$. Let $\rho$ denote the $(\emptyset, \emptyset)$-flag consisting of a single $2$-edge (this flag corresponds to the edge density). Let $f(\rho)$ be denote the upper bound on $\tau(\rho)$ given in Theorem~\ref{theorem: upper bounds on tau}.

The function $f(\rho)$ is piecewise linear, continuous and strictly increasing in the interval $(\frac{1}{2}, 1]$. In particular, it has a piecewise linear inverse. Over any subinterval $I\subseteq[\frac{1}{2}, 1]$ on which is $f$ is linear, we can use semidefinite method to obtain an upper bound on how much $\tau(\rho)$ can deviate from $f(\rho)$  on $I$ by  giving an upper bound for the following problem.
\begin{problem}\label{problem: defect}
Maximise $\rho-f^{-1}(y)$ over $y\in f(I)$ subject to the constraint $T_{1}\leq y$.
\end{problem}
Note the constraint we have given corresponds to requiring that all but $o(1)$ proportion of the vertices have triangle-degree at most $y\frac{n^2}{2}+o(n)$ (which is slightly weaker than what we require for $\tau$). A standard flag algebra computation will give us an upper bound $\varepsilon_I>0$ on the solution to Problem~\ref{problem: defect}. If $f(x)= a x +b$ over the interval $I$, then this tells us that $f(x-\varepsilon_I)= a(x-\varepsilon_I )+b$ is a lower bound for $\tau(x)$ on the interval $I-\varepsilon_I:=\{x\in I: \ x\leq \sup I -\varepsilon_I\}$, i.e that $f(\rho)$ is a most $a\varepsilon_I$ away from the true value of $\tau(\rho)$ on $I-\varepsilon$. Using this technique, we obtain the following:
\begin{theorem}\label{theorem: flag algebra bounds}
\[\tau(\rho) \geq \left\{\begin{array}{ll}
f(\rho)-0.0010705 & \textrm{if }\rho \in [\frac{1}{2}, \frac{29}{54}]\\ 
f(\rho)-0.0044863 & \textrm{if }\rho \in [\frac{29}{54}, \frac{31}{54}]\\
f(\rho)-0.0106917 & \textrm{if }\rho \in [\frac{31}{54}, \frac{11}{18}]\\
f(\rho)-0.0106917 & \textrm{if }\rho \in [\frac{11}{18}, \frac{17}{27}]\\
f(\rho)-0.0058198 & \textrm{if }\rho \in [\frac{17}{27}, \frac{35}{54}]\\
f(\rho)-0.0002057 & \textrm{if }\rho \in [\frac{35}{54}, \frac{2}{3}]\\
f(\rho)-0.00123143 & \textrm{if }\rho \in [\frac{2}{3}, \frac{25}{36}]\\
f(\rho)-0.00534603& \textrm{if }\rho \in [\frac{25}{36}, \frac{13}{18}]\\
f(\rho)-0.00534583 & \textrm{if }\rho \in [\frac{13}{18}, \frac{53}{72}]\\
f(\rho)-0.00189005 & \textrm{if }\rho \in [\frac{53}{72}, \frac{3}{4}]

\end{array} \right.\]	
\end{theorem}
\begin{proof}
	The theorem follows from standard algebra computations using the method outline above. Running the script \texttt{theorem38.sage} which is found in the  auxiliary files of this ArXiv submission on Flagmatic 2.0 yields the bounds claimed above. (The resulting computation certificates are somewhat large, but the computation itself can easily be run on a modern laptop computer.)
\end{proof}

We also `zoom in'  on the value $\rho=\rho_{\star}$ at which $\tau (\rho)$ becomes greater than $1-\rho$, and which we conjecture is equal to $\frac{19}{27}$. This is done by giving an upper bound for the following variant of Problem~\ref{problem: defect}:
\begin{problem}\label{problem: pointwise bound}
	Maximise $1-x -\rho$ subject to the constraint $T_{1}\leq x$.
\end{problem}
Suppose for some fixed $x$ we perform a flag algebra calculation and get a non-positive upper bound for the solution to Problem~\ref{problem: pointwise bound} . This implies that any $n$-vertex graph with at least $(1-x)\frac{n^2}{2}+o(n^2)$ edges must have a positive proportion of its vertices having triangle-degree greater than $x\frac{n^2}{2}+o(n^2)$. In particular we must have $\rho_{\star}\leq x$. Using this technique, we obtain the following bounds on $\rho_{\star}$
\begin{proposition}\label{prop: pointwise flag algebra bounds}
\[\rho_{\star}\leq \frac{19}{27}+7.4\times 10^{-9}.\]
\end{proposition}
\begin{proof}
	The theorem follows from standard algebra computations using the method outline above. Running the script \texttt{theorem310.sage} which is found in the  auxiliary files of this ArXiv submission on Flagmatic 2.0 yields the bounds claimed above. (This is a much smaller computation than the one required for Theorem~\ref{theorem: flag algebra bounds}.)
\end{proof}

\section{Concluding remarks}\label{section: concluding remarks}

In earlier sections we showed that 
\[
c_1(K_4^{(3-)})= \frac{\sqrt{13}-1}{6}, \quad \frac{19}{27} \le c_1(K_4^{(3)})\le \frac{19}{27} + 7.4\times 10^{-9}  , \quad \mathrm{and} \quad
\frac{5}{9}\leq c_1(C_5^{(3)})\leq 2-\sqrt{2}. 
\]
We conjecture that $c_1(K_4^{(3)})= \frac{19}{27}$ and $c_1(C_5^{(3)})=\frac{5}{9}$.

\subsection{Book numbers of graphs}\label{subsection: book numbers}
In Section~\ref{section: triangle-degree}, we investigated the following question: let $G$ be a graph on $n$ vertices with $m>\mathrm{ex}(n, K_3^{(2)})$ edges. What is the largest $t$ such that $G$ must have some \emph{vertex} contained in at least $t$ triangles? A different but equally natural question is to ask: what is the largest $b$ such that $G$ must have some \emph{edge} contained in at least $b$ triangles?  This is in fact a well-studied problem in graph theory.
\begin{definition}\label{def: tbook number}
	Let $G$ be a $2$-graph, and $xy\in E(G)$. The \emph{book size} of $xy$ in $G$ is
	$\mathrm{bk}(xy)=\mathrm{bk}_G(xy):= \vert \Gamma(x)\cap \Gamma(y)\vert$, 
	the number of triangles in $G$ containing the edge $xy$.The \emph{book number} of $G$ is
	\[\mathrm{bk}(G) :=\max\{\mathrm{bk}(xy): \ xy\in E(G)\}.\]
\end{definition}
The study of book numbers in graphs was initiated by Erd{\H o}s in 1962~\cite{Erdos62}, and has attracted considerable attention in extremal graph theory and Ramsey theory. Set 
\[\beta(n, m):=\min \{\mathrm{bk}(G): \ v(G)= n, \ e(G)=m\},\]
and  
\[\beta(x):=\inf_n \{\mathrm{bk}(G)/n: v(G)=n,  \ e(G)\geq  x{n \choose 2} \},\]

Erd{\H o}s conjectured that $\beta(n, \mathrm{ex}(n, K_3^{(2)}) +1 )>\frac{n}{6}$. This was proved by Edwards~\cite{Edwards77}  and independently by Khad{\v z}iivanov--Nikiforov~\cite{KhadziivanovNikiforov79}. Bollob\'as and Nikiforov~\cite{BollobasNikiforov05} determined $\beta(n,m)$ exactly for infinitely many value of $m$ with $\frac{n^2}{4}<m < \frac{n^2}{3}$. 

A construction giving the best known lower bound on $\beta(n,m)$ was given by Erd{\H o}s, Faudree and Gy\"ori ~\cite{ErdosFaudreeGyori95}, generalising an earlier construction due to Erd{\H os}, Faudree and Rousseau~\cite{ErdosFaudreeRousseau94}.
\begin{construction}[Erd{\H os}, Faudree and Gy\"ori~\cite{ErdosFaudreeGyori95}]\label{construction: even better extremal graphs for book number} Suppose $n=r_1\cdot r_2\cdot r_3\cdots r_{k-1} \cdot r_kt$, where $r_1, r_2, \ldots, r_k, t$ are strictly positive integers satisfying $(r_{i-1}-1)^2< r_i$ for every $i\in [k]$. Set 
	\[V=\{ (i_1, i_2, \ldots , i_k, i_{k+1}): \ i_j \in [r_j] \textrm{ for all } j \in [k], \ i_{k+1} \in [t]\}.\] 
	Define a graph $G$ on $V$ by joining pairs of vectors from $V$ by an edge if and only if they differ in each of the first $k$ coordinates.
\end{construction}
This construction gives rise to a $d$-regular graph with book number $b$, where $d=\prod_{i=1}^k \left(\frac{r_i -1}{r_i}\right)n$ and $b=\prod_{i=1}^k \left(\frac{r_i -2}{r_i}\right)n$. Erd{\H o}s, Faudree and Gy\"ori  conjectured this gives the correct behaviour for the minimum value of the book number in graphs subject to a minimum degree condition.
\begin{conjecture}[Erd{\H o}s, Faudree and Gy\"ori~\cite{ErdosFaudreeGyori95}]
	\label{EFG conjecture: minimum book number}
	Let $x \in \mathbb{Q}$ with $\frac{1}{2}<x<1$. Let 
	\[x= \prod_{i=1}^k \frac{r_i-1}{r_i}\]
	with $3\leq r_1$ and $(r_{i-1}-1)^2<r_i$ for $2\leq i \leq k$ be the (unique) ``greedy representation'' of $x$. Set 
	\[b(x)= \prod_{i=1}^k \frac{r_i-2}{r_i}.\]
	Then every graph on $n$ vertices with minimum degree $d\geq xn$ has book number at least $b(x)n$.	
\end{conjecture}

We believe that the minimum degree condition in Conjecture~\ref{EFG conjecture: minimum book number} can be replaced by a size condition, and this belief seemed to be borne out by flag algebra computations we ran for this problem.
\begin{conjecture}\label{conjecture: exact size for book number in terms of edge number}
	Let $x \in \mathbb{Q}\cap (\frac{1}{2}, 1)$ and $b(x)$ be  as above. Then  $\beta(x)=b(x)$, i.e. any graph on $n$ vertices with at least $x\frac{n^2}{2}$ edges has book number at least $b(x)n$.
\end{conjecture}

\subsection{Maximal triangle-degree, book number and triangle density}
In Sections~\ref{section: triangle-degree} and ~\ref{subsection: book numbers}, we discussed the maximum triangle-degree of a vertex and the book number (i.e. maximum triangle-degree of an edge) in graphs, giving conjectures on their minimum value for a given edge-density or minimum degree condition. Here we compare the conjectured behaviour of these two triangle-related extremal quantities with each other and with the minimal triangle density in graphs $G$ with $\rho \binom{n}{2}$ edges for $\frac{1}{2}\leq \rho \leq \frac{2}{3}$.

Razborov~\cite{Razborov08} showed that such a graph $G$ contain at least $\kappa(\rho)\binom{n}{3}+o(n^3)$ triangles, where
\[\kappa(\rho) =\frac{1}{6} \left(1- \sqrt{2(2-3\rho)}\right) \left(2+\sqrt{2(2-3\rho)}\right)^2.\]
In addition, Lo~\cite{Lo12} showed that if the minimum degree of $G$ is at least $\rho n$, then it contains at least $\lambda(\rho)\binom{n}{3}+o(n^3)$ triangles, where 
	\[\lambda(\rho) =   3\rho (1-\rho)(2\rho -1).   \]	
Conjecture~\ref{conjecture: tau value} implies that every $n$-vertex graph with edge density $\rho$ contains a vertex with triangle-degree at least $\tau'(\rho)\binom{n}{2}+o(n^2)$, where
\[\tau'(\rho)= \left\{ \begin{array}{ll}
\frac{3}{2}\left(\rho-\frac{1}{2}\right) & \textrm{if }   \frac{1}{2}\leq \rho\leq \frac{11}{18},\\
\rho-\frac{4}{9} & \textrm{if }   \frac{11}{18}\leq \rho\leq \frac{2}{3}.
\end{array}
\right. \]
Finally, let $\beta'(x)$ denote  the function obtained by extending the function $b(x)$ from Conjecture~\ref{EFG conjecture: minimum book number} from the rationals in $(\frac{1}{2}, \frac{2}{3}]$ to a monotonically increasing left-continuous function on the whole interval. 
This last function unfortunately does not have a nice closed form, but we can plot an approximation of it (or rather: $\rho \beta'(\rho)$) along the other three in Figure~\ref{plot1}, allowing for a visual comparison of the four functions $\kappa$, $\lambda$, $\tau'$ and $\rho \beta'$ in the interval $\rho\in (\frac{1}{2}, \frac{2}{3}]$. 
\begin{figure}\label{figure}
	\begin{center}
		\includegraphics[width=0.75\textwidth]{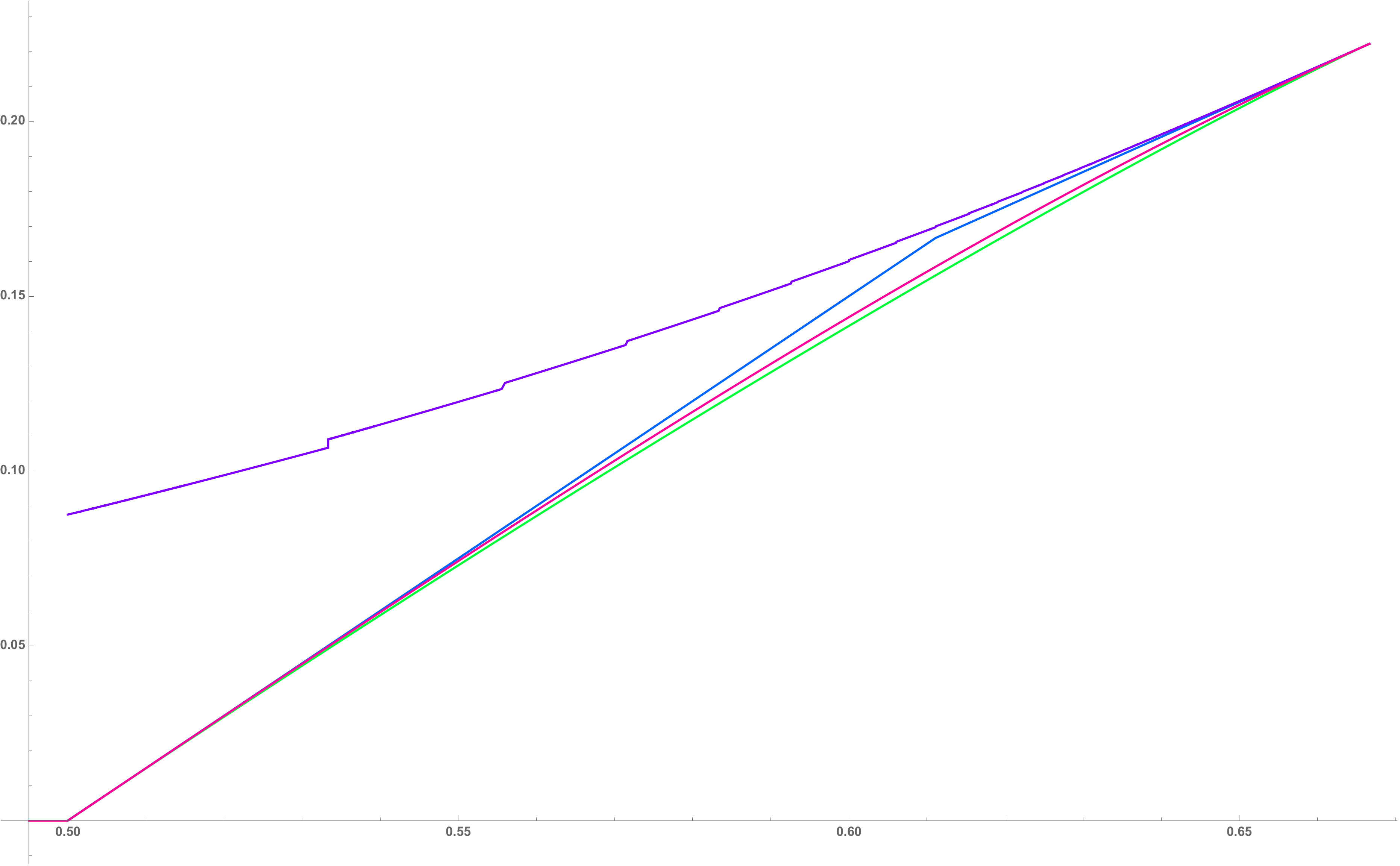}
	\end{center}
	\caption{The functions $\rho\cdot \beta'(\rho)$, $\tau'(\rho)$, $\lambda(\rho)$, $\kappa(\rho)$ (from top to bottom)
	}\label{plot1}
\end{figure}

Clearly by averaging we have that $\kappa(\rho)$ is the smallest of the functions in Figure~\ref{plot1}.  Assume now that Conjecture~\ref{conjecture: tau value} is true. Then Constructions~\ref{construction: upwards} and \ref{construction: downwards} provide $\rho n$-regular graphs $G$ of order $n$ with $t_{\mathrm{max}}(G)= \tau(\rho) n^2/2 +o(n)$. Averaging the triangle-degree over all vertices, this would give that $\lambda(\rho)\leq \tau(\rho)$.

Further assuming Conjecture~\ref{EFG conjecture: minimum book number} is true,  Construction~\ref{construction: even better extremal graphs for book number} gives a $\rho n$-regular graph $G$ of order $n$ with book number $\beta n$ that is triangle-degree regular with  $t_{\mathrm{max}}(G)= \rho \beta (\rho)\binom n2 + o(n^2)$. This would imply that $\tau(\rho)\leq \rho \beta(\rho)$, and all together, 
\begin{align}\label{inequality: four functions}
\kappa(\rho)  \leq \lambda(\rho)  \leq \tau(\rho) \leq \rho\beta(\rho). 
\end{align}
In Figure~\ref{plot1} we plotted the four functions $\kappa(\rho)$, $\lambda(\rho)$, $\tau'(\rho)$ and $\rho \beta'(\rho)$ in the interval $[\frac{1}{2}, \frac{2}{3}]$. As the plot shows, the inequalities in (\ref{inequality: four functions}) with $\tau'$ and $\beta'$ taking the place of $\tau$ and $\beta$ all hold in $[\frac{1}{2}, \frac{2}{3}]$ with equality if and only if $\rho= \frac{1}{2}$ (for the first two inequalities) or $\frac{2}{3}$ (for all three).

\section*{Acknowledgements}
The authors gratefully acknowledge the support of a Wenner--Gren guest professorship awarded to the third author to visit Ume{\aa} in May--June 2017, when we conducted this research. The first author would also like to thank the earlier support of an AMS-Simons grant for a visit to Atlanta in Spring 2016, in the course of which he and the third author began discussing vertex-degree thresholds for covering. Victor Falgas--Ravry's research is supported by VR grant 2016-03488. Klas Markstr\"om's research is supported by VR grant 2014-4897. Yi Zhao's research is supported by NSF grants DMS-1400073 and  DMS-1700622.

\bibliographystyle{plain}

\end{document}